\documentclass[oneside]{amsart}

\usepackage{amsmath}
\usepackage{mathtools}
\usepackage{amssymb}
\usepackage{amsthm}
\usepackage{color}
\usepackage{graphicx}
\usepackage{tikz-cd}
\usepackage{tikzoverlay}
\usetikzlibrary{calc}
\usetikzlibrary{decorations.pathreplacing}
\usepackage{xcolor}
\usepackage{aliascnt}
\usepackage[margin=1in]{geometry}
\usepackage[hidelinks]{hyperref}
\usepackage[alphabetic,msc-links]{amsrefs}
\usepackage[capitalize]{cleveref}
\usepackage{float}
\usepackage{subcaption}
\usepackage{comment}

\definecolor{figred}{cmyk}{0,.8,1,0}
\definecolor{figyellow}{cmyk}{.1,.05,.9,0}
\definecolor{figorange}{cmyk}{0,.5,1,0}
\definecolor{figblue}{cmyk}{1,.5,0,0}
\definecolor{figmagenta}{cmyk}{.1,.7,0,0}
\definecolor{figcyan}{cmyk}{.8,0,0,0}
\definecolor{figgreen}{cmyk}{.97,0,.75,0}

\let\oldtheorem\newtheorem
\RenewDocumentCommand{\newtheorem}{s m o m O{}}{\IfBooleanTF{#1}{\oldtheorem*{#2}{#4}}{\IfNoValueTF{#3}{\oldtheorem{#2}{#4}[#5]}{\newaliascnt{#2}{#3}\oldtheorem{#2}[#2]{#4}\aliascntresetthe{#2}}}}

\makeatletter
\newtheorem*{rep@theorem}{\rep@title}
\newcommand{\newreptheorem}[2]{\newtheorem*{rep@#1}{\rep@title}\newenvironment{rep#1}[1]{\def\rep@title{#2 \ref*{##1}}\begin{rep@#1}}{\end{rep@#1}}}
\makeatother

\newtheorem{maintheorem}{Theorem} 
\newreptheorem{maintheorem}{Theorem}

\newtheorem{maincorollary}[maintheorem]{Corollary} 
\newreptheorem{maincorollary}{Corollary}

\newtheorem{theorem}{Theorem}[section]
\newtheorem{lemma}[theorem]{Lemma}
\newtheorem{proposition}[theorem]{Proposition}
\newtheorem{corollary}[theorem]{Corollary}

\theoremstyle{definition}

\theoremstyle{remark}
\newtheorem{remark}[theorem]{Remark}

\DeclareMathOperator{\Sph}{\mathcal{S}}
\DeclareMathOperator{\cliq}{\mathcal{K}}
\DeclareMathOperator{\Out}{\mathrm{Out}}
\DeclareMathOperator{\Aut}{\mathrm{Aut}}
\DeclareMathOperator{\IO}{\mathrm{IO}}
\DeclareMathOperator{\Mod}{\mathrm{Mod}}

\DeclareMathOperator{\Twist}{\mathrm{Twist}}
\DeclareMathOperator{\CD}{\mathrm{CD}}
\DeclareMathOperator{\HD}{\mathrm{HD}}
\DeclareMathOperator{\Hom}{\mathrm{Hom}}

\title{Coloring graphs of homologous spheres}
\address{San Jos\'{e} State University, One Washington Square, San Jose, CA, 95112}
\author{Edgar A. Bering IV}
\email{edgar.bering@sjsu.edu}
\author{Samantha Ford}
\email{samantha.ford@sjsu.edu}
\author{Aarav Ghai}
\email{aaravmanav.ghai@sjsu.edu}
\author{Maxwell Lyon}
\email{maxwell.lyon@sjsu.edu}
\author{Om Shah}
\email{om.shah@sjsu.edu}

\begin{document}

\begin{abstract}
The curve graph of a surface is the graph whose vertices are isotopy classes of essential, closed curves on the surface.  Edges join vertices with disjoint representatives. The curve graph plays a central role in the theory of mapping class groups of surfaces. Gaster, Greene, and Vlamis have shown that the subgraph of the curve graph induced by curves in a single primitive homology class is uniquely and finitely colorable. In the analogy between the theory of mapping class groups of surfaces and outer automorphisms of free groups, the sphere graph of $M_r$, a connect sum of $r$ copies of $S_1 \times S_2$ plays an analogous role in the theory. Motivated by Gaster, Green, and Vlamis' result, we investigate $\Sph_v(M_r)$, the subgraph of the sphere graph of $M_r$ induced by spheres representing a single primitive homology class $[v] \in H_2(M_r;\mathbb{Z})$. We show that this graph is uniquely $r-1$ colorable. As a corollary, with $r\ge 3$ we obtain a nontrivial homomorphism from the Torelli subgroup $IO_r\le \Out(F_r)$ to the symmetric group on $r-1$ symbols; this parallels Gaster, Greene, and Vlamis' connection between their coloring and the Chillingworth homomorphism. In the course of our proof we find new smaller generating sets for mapping class groups of $M_r$ with spheres deleted. 
\end{abstract}

\maketitle

\section{Introduction}

The sphere graph of a connect sum of $r$ copies of $S^1\times S^2$, isomorphic to the graph of one-edge free splittings of $F_r$~\cite{aramayona-souto}*{Lemma 2}, plays an important role in the theory of outer automorphism groups of free groups---analogous to the role of the curve graph of a surface. The study of the large scale geometry of this graph has led to significant progress in the area~\cites{hm1,hm2,bf-ff,subfactor}. In addition to this role the sphere graph is a rich combinatorial object, and there is increasing interest in its graph theoretic properties, including its chromatic number~\cites{aramayona-souto,bering-leininger,bering-colors}.

In the surface setting, Gaster, Greene, and Vlamis established the best known upper bound on the chromatic number of the curve graph of a surface~\cite{gaster}*{Theorem 1.5}. A key component of their argument studies the subgraph of the curve graph induced by a single primitive homology class. They find that this subgraph has a \emph{unique} proper coloring realizing its chromatic number~\cite{gaster}*{Theorem 1.4}. Following the long running analogy between mapping class groups and outer automorphism groups, we obtain an analogous unique coloring result for spheres in a single primitive homology class.

\begin{maintheorem}\label{main theorem}
Let $[v] \in H_2(M_r)$ be a primitive homology class, with $r\ge 2$. The graph $\Sph_v(M_r)$ of spheres in the class $[v]$ is uniquely $r-1$ colorable.
\end{maintheorem}

\begin{proof}
\cref{existence proposition} establishes that the chromatic number of $\Sph_v(M_r)$ is $r-1$, by constructing a coloring equal to the size of a maximal clique. \cref{uniqueness proposition} shows this coloring is unique up to permutation of colors, by using Putman's lemma (\cref{Putman connectivity}) to establish the connectivity of the graph of homologous spheres and its clique graph.
\end{proof}

Gaster, Greene, and Vlamis connect this unique coloring to an interpretation of the \emph{Chillingworth homomorphism}~\cite{gaster}*{Theorem 1.6}. Fix a closed oriented surface $S$ of genus $g\ge 2$ and consider \[\mathcal{I} = \ker(\Mod(S)\to H_1(S,\mathbb{Z})),\] the \emph{Torelli group} of mapping classes that act trivially on homology. For a fixed primitive homology class $[c]$, $\mathcal{I}$ acts on the subgraph of the curve graph induced by all simple closed curves representing $[c]$. As the graph is uniquely properly colorable, this induces an action of $\mathcal{I}$ on the color set by permutation. Gaster, Greene, and Vlamis show that this action is cyclic and equal to the \emph{Chillingworth homomorphism}, $t : \mathcal{I} \to \Hom(H_1(S,\mathbb{Z}), \mathbb{Z}/(g-1)\mathbb{Z})$; described here as the mod $(g-1)$ reduction of the Johnson homomorphism~\cites{chillingworth-i,chillingworth-ii,johnson-hom}.

The analog of the Torelli group in the $\Out(F_r)$ setting is the subgroup $\IO_r$, the kernel of the natural action on $H_1(F_r,\mathbb{Z})$. (Note that $IO_r$ is often called $IA_r$ elsewhere in the literature, while some authors reserve the latter for the corresponding subgroup of $\Aut(F_r)$; we agree with the second convention.) As a corollary of our result, we obtain a homomorphism from $\IO_r$ to the symmetric group. 

\begin{maincorollary}\label{main corollary}
If $r\ge 3$, there is a nontrivial homomorphism
\[ \xi_v \colon \IO_r \to S_{r-1}. \]
\end{maincorollary}

\begin{proof}
The action of $\IO_r$ is trivial on $H_2(M_r)$. Thus $\IO_r$ acts naturally on $\Sph_v(M_r)$, and by the uniqueness of the proper $r-1$ coloring this induces an action of $\IO_r$ on the set of colors. Moreover, the action on $\Sph_v(M_r)$ is transitive. This is implicit in Landgraf~\cite{landgraf}*{Lemma 7.3}; we record the proof explicitly as \cref{homologous sphere graph connected}. Since $r\ge 3$ there is more than one color class. Thus there is some element of $\IO_r$ that moves a sphere from one color class to another. Therefore the homomorphism is nontrivial.
\end{proof}

As is often the case, pursuing the analogy is not a direct translation. To interpret the image of the color permutation homomorphism in the surface setting, Gaster, Greene, and Vlamis rely heavily on surface tools not immediately applicable to the 3--manifold setting. Nevertheless our result raises the question of whether there is a meaningful analog of the Chillingworth homomorphism in the $\Out(F_r)$ setting. Pursuing this direction will be the content of the second author's Master's thesis.

The proof of \cref{main theorem} relies on understanding the connectivity of the graph of spheres in a single homology class, as well as its maximal clique graph. We establish this connectivity using Putman's lemma~\cite{Putman}*{Lemma 2.1}, which has been applied to many variations of the curve graph. To use Putman's lemma one needs a group action with a well-understood generating set. For mapping class groups of surfaces, this is a mature subject. For our purposes, we develop some further details about generators for $\Mod(M_{r,k})$, the mapping class group of a connect sum of $r$ copies of $S^1\times S^2$ with $k$ disjoint open balls removed, in \cref{generators section}. Our analysis builds on work of Landgraf~\cite{landgraf} and Brendle, Broaddus, and Putman~\cite{BrendleBroaddusPutman} in describing generating sets of these groups. The groups $\Mod(M_{r,k})$ arise in the study of homological stability for $\Aut(F_r)$ and $\Out(F_r)$: they provide 3--manifold models of the automorphism groups of graphs with boundary first introduced by Wahl~\cites{HatcherWahl,wahl,jensen-wahl,hatcher-vogtmann}.

As mentioned in the opening paragraph, the sphere graph is isomorphic to the free splitting graph. In their work on the automorphism group of the sphere graph, Aramayona and Souto describe how to prove their main theorem entirely from the point of view of free splittings, without using 3--manifold topology~\cite{aramayona-souto}*{Postscriptum}. They also remark ``by interpreting free splittings as isotopy classes of spheres, one can think of free splittings of free groups almost in the same way as one thinks of elements in the curve complex of a surface.'' The subgraph of the sphere graph induced by a single primitive homology class can also be defined entirely in terms of free splittings. Any one-edge free splitting determines a primitive cohomology class $[\alpha] \in H^1(F_r,\mathbb{Z})$ by taking signed intersection with the midpoint of the edge; indeed, this class is Poincaré dual to the homology class of the sphere $v$ representing the splitting. The subgraph of the free splitting graph induced by the splittings that induce the cohomology class $[\alpha]$ is isomorphic to $\Sph_v(M_r)$. However, unlike the manipulations in Aramayona and Souto's paper, it is not clear to us how to reframe the sphere manipulations in this work in purely algebraic terms. In particular, understanding the actions used to establish connectivity from the free splitting perspective seems quite challenging: We have found it much easier to study the action of a diffeomorphism on a sphere than of a graph homotopy equivalence on a graph of groups, and to locate a disjoint sphere in an appropriate diagram than to find an intermediate graph of groups that shares a common refinement with two given graphs of groups.

\subsection*{Notational conventions and terminology} A \emph{proper vertex coloring} (or just \emph{proper coloring}) of a graph $G$ is a surjective function $\phi \colon V(G) \to X$ such that for every edge $(u,v)\in E(G)$, $\phi(u)\neq \phi(v)$. The cardinality $|X|$ is the size of the coloring. The \emph{chromatic number} of $G$, denoted $\chi(G)$, is the minimal size of a proper vertex coloring of $G$. A graph is \emph{uniquely colorable} for a proper coloring set $X$ if for any two proper $X$-colorings $\phi, \psi \colon V(G) \to X$, there is a permutation $\sigma \in S_X$ such that $\phi = \sigma\circ\psi$. Given a subset $W \subset V(G)$ the \emph{induced subgraph of $G$ on $W$} is the graph $H$ with $V(H) = W$ and an edge $(u,v) \in E(H)$ for $u,v\in W$ if and only if $(u,v) \in E(G)$.

For $r \ge 0$ we define $M_r = M_{r, 0} = \sharp_r S^1\times S^2$ to be the connect sum of $r$ copies of $S^1\times S^2$ with the convention $M_0 = S^3$. For $n\ge 0$ we define $M_{r,n} = M_r \setminus \sqcup_n B$ to be $M_r$ with $n$ disjoint open balls removed. A sphere embedded in a 3--manifold is \emph{essential} if it does not bound a ball and \emph{non-peripheral} if it is not isotopic to a boundary component. The \emph{sphere graph} of a 3--manifold M is the graph $\Sph(M)$ with vertex set the isotopy classes of embedded essential non-peripheral spheres in $M$. Two classes are joined by an edge if they have disjoint representatives. For brevity, we will refer to vertices of $\Sph(M)$ as \emph{spheres in $M$} when no ambiguity occurs. Any primitive homology class $[v]\in H_2(M_r,\mathbb{Z})$ is represented by an embedded essential sphere $v$~\cite{landgraf}*{Lemma B.2}, and we will consider the subgraph induced by spheres homologous to $v$, denoted $\Sph_v(M_r)$ and referred to as the \emph{graph of homologous spheres} for brevity. The choice of primitive homology class $[v]$ gives an orientation to every vertex of $\Sph_v(M_r)$ and we adopt the convention that all representatives of vertices are oriented with the orientation given by $[v]$.

For an embedded essential sphere $v$ in $M$, the complement $M \setminus N(v)$ for a regular of $v$ that deformation retracts onto $v$ is referred to as the result of \emph{cutting $M$ along $v$}. When specifying a cut we use $N(v)$ to mean an appropriate neighborhood. This notion naturally extends to finite collections of disjoint embedded spheres. The result $M\setminus N(v)$ has two new boundary components, the orientation of $v$ identifies the boundaries with labels $v^-$ and $v^+$, and $M$ is diffeomorphic to the quotient of the cut manifold identifying $v^-$ and $v^+$.

\subsection*{Graphical conventions for \texorpdfstring{$M_{r,n}$}{M r,n}} We use two graphical conventions to diagram spheres inside $M_{r,n}$, depending on expositional needs. A \emph{cut-system} in $M_{r,n}$ is a disjoint union of embedded spheres $C$ such that $M_{r,n}\setminus N(C) = M_{0,n+2r}$ for a sufficiently small regular neighborhood $N(C)$. The manifold $M_{r,n}$ is recovered from $M_{0,n+2r}$ by identifying the pair of boundary spheres in $M_{0,n+2r}$ corresponding to each sphere in $C$. A \emph{cut-sphere diagram} of $M_{r,n}$ is a stereographic projection of $S^3\setminus\{n+2r\text{ open balls}\}$. In this diagram boundary spheres are drawn as spheres, with pairs of labels indicating the identification of the cut-system. Spheres other than the cut system are drawn either as embedded spheres in $M_{0,n+2r}$ or as a collection of disks and cylinders with boundaries on the cut-system spheres; identification labels are omitted when clear from context. See \cref{cut-sphere-sample} for an example. The second graphical convention is a \emph{handlebody diagram}. The manifold $M_r$ has a standard Heegaard splitting into two genus $r$ handlebodies. Spheres are drawn as collections of disks and cylinders embedded in the two handlebodies. Often, the diagram in each handlebody is identical. In this case we show only one and the manifold and embedded spheres are obtained as the appropriate doubles. To draw a handlebody diagram of $M_{r,n}$, boundary spheres are either shown as embedded in one of the handlebodies for $M_{r,0}$, or as disks on the boundary of the handlebody, indicating a relative Heegaard splitting. See \cref{handle-diagram-sample} for an example.

\begin{figure}
\centering
\begin{subfigure}[c]{0.49\textwidth}
\begin{tikzoverlay}[scale=1.0]{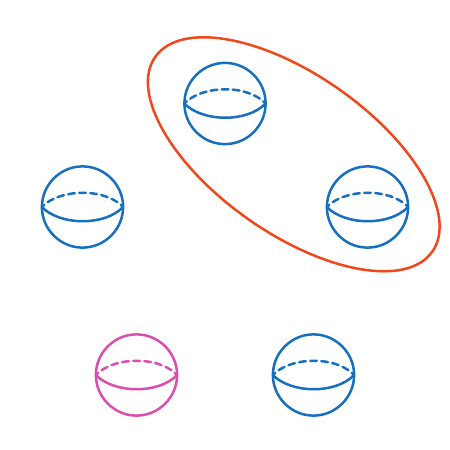}[font=\small]
    \node[figblue] at (19.2,48.1) {$a^+$};
    \node[figblue] at (51.0,71.1) {$a^-$};
    \node[figblue] at (82.7,48.1) {$b^+$};
    \node[figblue] at (71,10.6) {$b^-$};
    \node[figmagenta] at (30.6,10.6) {$\partial$};
\end{tikzoverlay}
\caption{The cut-sphere diagram.}
\label{cut-sphere-sample}
\end{subfigure}
\begin{subfigure}[c]{0.49\textwidth}
\begin{tikzoverlay}[scale=1.0]{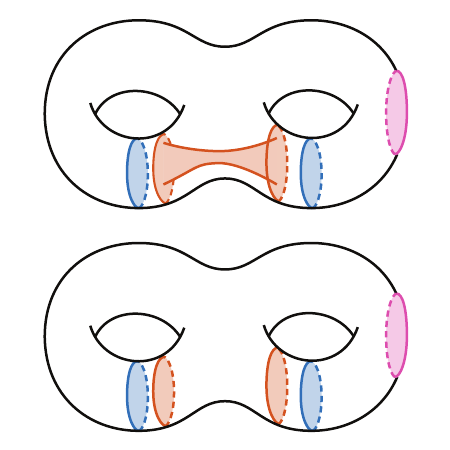}[font=\small]
    \node[figmagenta] at (88.0,75.5) {$\partial$};
    \node[figblue] at (30.3,50.5) {$a$};
    \node[figblue] at (69.2,50.5) {$b$};
    \node[figblue] at (30.3,1.5) {$a$};
    \node[figblue] at (69.0,1.5) {$b$};
\end{tikzoverlay}
\caption{The handlebody diagram.}
\label{handle-diagram-sample}
\end{subfigure}
\caption{A cut-system (blue), embedded essential sphere (red), and boundary sphere (purple) for $M_{2,1}$, depicted using both graphical conventions.}
\end{figure}

\subsection*{Acknowledgments} The authors were supported by the National Science Foundation under Award No. DMS--2532818. We EB thank Ric Wade for an enlightening exchange about crossed homomorphisms and the possibility of a Chillingworth homomorphism for $IO_r$, as well as helpful comments on an earlier draft of the manuscript.

\subsection*{Tool Use} The authors used Hagaromo chalk, slate, pencil, paper,
and a Hewlitt Packard printer in the inital stages of the investigation. The
manuscript was typeset using Overleaf's LaTeX platform. Figures were prepared
with Adobe Illustrator and Nathan Dunfield's labelpin. No other sophisticated
tools (including LLMs) were used.

\section{The chromatic number of homologous spheres}\label{existence section} 

We begin by investigating the chromatic number of the graph of homologous spheres, determining it exactly.

\begin{proposition}\label{existence proposition} Suppose $r\ge 2$ and $v$ is an embedded sphere in $M_r$ such that $[v]\in H_2(M_r)$ is primitive. Then
\[ \chi(\Sph_v(M_r)) = r-1. \]
\end{proposition}

The proof is divided into two lemmas: \cref{lower bound} computes a lower bound and \cref{upper bound} exhibits a coloring realizing this bound.

\begin{lemma}\label{lower bound}
Assume $r \ge 2$ and $v$ is an embedded sphere in $M_r$ such that $[v]\in H_2(M_r)$ is primitive.  Then
\[ \chi(\Sph_v(M_r)) \ge r-1. \]
\end{lemma}

\begin{proof}

We exhibit a clique of size $r-1$ in $\Sph_v(M_r)$. Since $\chi$ is greater than or equal to the size of the largest clique for every graph, this establishes the bound. 

Since $[v]\neq 0$, the sphere $v$ is nonseparating. Cutting $M_r$ along $v$ produces a manifold
$
N\cong M_{r-1,2}.
$
Denote the two boundary components of $N$ by $v^-$ and $v^+$. The identification of the boundary gives a quotient $N \to M_r$ that is a diffeomorphism on the interior of $N$.

Choose a decomposition
\[
N=A_1\cup A_2\cup\cdots\cup A_{r-1},
\]
where each $A_i\cong M_{1,2}$, the intersection of consecutive $A_i$ is an embedded sphere,  $A_1\cap \partial N = v^-$  and $A_{r-1}\cap\partial N = v^+$.
For $1\leq i\leq r-2$, let
$
S_i=A_i\cap A_{i+1}.
$
This decomposition is illustrated in Figure~\ref{fig:block-decomposition}.

We identify each $S_i$ with its image in the identification $N \to M_r$. In $M_r$ the spheres
$
v,S_1,\ldots,S_{r-2}
$
are pairwise disjoint and non-isotopic.

\begin{figure}
    \centering
\begin{tikzoverlay}[scale=1.0]{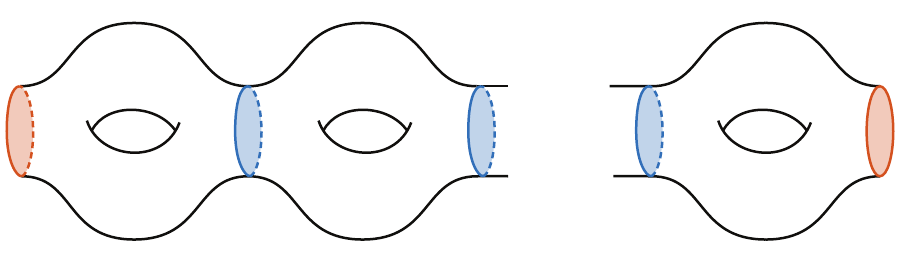}[font=\small]
    \node[figred] at (2.8,7.2) {$S_0^-$};
    \node[] at (15.2,1.0) {$A_1$};
    \node[figblue] at (28.2,7.2) {$S_1$};
    \node[] at (40.5,1) {$A_2$};
    \node[figblue] at (53.8,7.2) {$S_2$};
    \node[font=\large] at (63.4,14.6) {$\cdots$};
    \node[figblue] at (72.5,7.2) {$S_{r-2}$};
    \node[] at (84.9,1) {$A_{r-1}$};
    \node[figred] at (98.7,7.2) {$S_0^+$};
\end{tikzoverlay}
    \caption{A handlebody diagram of the decomposition of $M_{r-1,2}$ into $r-1$ copies of
    $M_{1,2}$, denoted $A_i$. Identifying $S_0^-$ with $S_0^+$ recovers $M_r$.}
    \label{fig:block-decomposition}
\end{figure}

Orient these spheres so that 
$
\partial A_i=S_i-S_{i-1}
$
for $1\leq i\leq r-1$, where we adopt the convention $S_{r-1}=S_0=v$. Since $\partial A_i$ is null-homologous in $M_r$, we have
$
[S_i]=[S_{i-1}]
$
for every $i$. It follows that
$
[S_i]==[v]
$
for all $0\leq i\leq r-2$. In particular, every $S_i$ is essential.

Since the $S_i$ are essential, pairwise disjoint, non-isotopic, and in the class $[v]$,
$
\{v,S_1,\ldots,S_{r-2}\}
$
determines a clique with $r-1$ vertices in $\Sph_v(M_r)$. Every proper coloring assigns distinct colors to the vertices of a clique, and therefore
\[
\chi(\Sph_v(M_r))\geq r-1.\qedhere
\]
\end{proof}

Constructing the coloring is more involved. Taking a cue from Gaster, Greene, and Vlamis we investigate immersions between homologous spheres. Suppose $a, b$ are two oriented spheres in $M_r$, define \emph{an immersion from $a$ to $b$ (of rank $s$)} to be an immersion $\iota \colon M_{s, 2}\looparrowright M_r$ such that restricted to the oriented boundary $\iota(\partial M_{s,2}) = -a \cup b$, where the minus sign denotes opposite orientation. The rank of an immersion between spheres, mod $r-1$, will serve as our color set.

\begin{lemma}\label{cobounding-lifts} 
    Let $a,b\subset M_r$ be embedded spheres with $[a]\in H_2(M_r)$ primitive. Let $p_a:M_\infty\to M_r$ be the infinite cyclic cover corresponding to algebraic intersection with $a$. Suppose $\widetilde{a}$ and $\widetilde{b}$ are disjoint lifts of $a$ and $b$. If $a$ and $b$ are homologous, then $\widetilde{a}$ and $\widetilde{b}$ cobound a compact submanifold of $M_\infty$. 
\end{lemma}

\begin{proof}
    Abstractly, the cover $p_a$ is the cover associated to the kernel of the cohomology class Poincar\'e dual to $[a]$. It can be realized concretely by cutting $M_r$ open along $a$ and stacking $\mathbb{Z}$ copies of $M_r \setminus N(a)$. Fix a lift $\widetilde{a}$ to $M_\infty$. This lift is one of the separating spheres between two consecutive copies of $M_r \setminus N(a)$, so $\widetilde{a}$ separates $M_\infty$ into two components, $M^\pm_\infty$.
    
    We now prove the contrapositive of the statement. Suppose $\widetilde{a}$ and $\widetilde{b}$ do not cobound a compact submanifold of $M_\infty$. We claim there is a closed curve $\widetilde{\gamma}\in X$ which is disjoint from $\widetilde{a}$, intersects $\tilde{b}$ once algebraically, and does not intersect any other lift of $\tilde{b}$. 
    Since $\widetilde{a}$ and $\widetilde{b}$ do not cobound a compact submanifold of $M_\infty$, both boundary components corresponding to $\widetilde{b}$ in $M_\infty \setminus N(\widetilde{a}\cup\widetilde{b})$ must be in the same connected component, $X_0$. Further, since $M_\infty$ is a regular cover and $b$ is disjoint from $a$, the lift $\widetilde{b}$ is contained in a lift $X\subset X_0$ of $M_r\setminus N(a)$, which is connected. Take a neighborhood around the lift $\widetilde{b}$ and points $x^+, x^- \in N(\widetilde{b})$ on opposite sides of $\widetilde{b}$ contained in $X$. There is path in $X$ from $x^+$ to $x^-$, joining $x^+$ to $x^-$ via an arc in $N(\widetilde{b})$ crossing $\widetilde{b}$ once gives the desired curve $\widetilde{\gamma} \subset X$. Indeed, every other lift of $b$ differs by a deck transformation, and so is disjoint from $X$.
    
    Project $\widetilde{\gamma}$ to a closed curve $\gamma=p_a(\widetilde{\gamma})\subset M_r$. Since $\widetilde{\gamma}$ is closed in the infinite cyclic cover, $\gamma$ has algebraic intersection number $\iota(\gamma, a)=0$. On the other hand, $\widetilde{\gamma}$ intersects exactly one lift of $b$ once, so its projection satisfies $\iota(\gamma,b)=\pm 1$. Since $a$ and $b$ have different algebraic intersection numbers with the closed curve $\gamma$, $[a]\ne [b] \in H_2(M_r;\mathbb Z)$. This proves the contrapositive.
\end{proof}

    \begin{figure}
    \centering
    \begin{tikzoverlay}[scale=1.0]{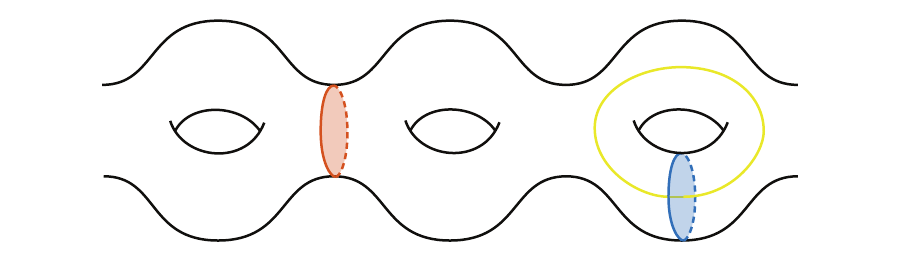}[font=\small]
    \node[figred] at (37.4,8.1) {$\tilde{a}$};
    \node[font=\large] at (8.5,14.5) {$\cdots$};
    \node[font=\large] at (92.5,14.5) {$\cdots$};
    \node[figblue] at (75.9,0.8) {$\tilde{b}$};
    \node[figyellow] at (85.9,13.1) {$\gamma$};
    \draw[very thick,decorate,decoration={brace,amplitude=8pt}] (37.4,27.4) -- (88.5,27.4) node[midway,black,yshift=16pt] {$X$};
\end{tikzoverlay}
    \caption{The handlebody diagram of $M_\infty$ indicating the image of the component $X$ under the natural gluing. The gluing of $X$ contains
    $\widetilde{b}$ and the curve $\widetilde{\gamma}$ crossing
    $\widetilde{b}$ once.}
    \label{fig:cobounding-lifts}
\end{figure}

\begin{corollary}\label{immersions exist}
If $a$ and $b$ are embedded spheres in $M_r$ representing the same primitive class $[v] \in H_2(M_r;\mathbb Z)$, then there exists an immersion from $a$ to $b$.
\end{corollary}

\begin{proof}

    Since $[b] = [a] = [v]$, \cref{cobounding-lifts} implies that we can choose a lift of $b$, $\widetilde b\subset M_\infty^+$ cobounding a compact sub-manifold $T\subset M_\infty^+$ with $\widetilde a$.

    We orient $\widetilde b$ such that $\partial T = -\widetilde a \cup \widetilde b$. Since $T$ is a compact sub-manifold with 2 boundaries, $T \cong M_{h,2}$, with $h \ge 0$.
    The restriction $p_v|_T: T \to M_r$
    maps one oriented boundary sphere to $-a$ and the other to $b$. Therefore $p_v|_T$ is an immersion from $a$ to $b$ of rank $h$.
\end{proof}

\begin{lemma}\label{well defined rank between}
Suppose $r \ge 2$ and $a,b$ are embedded spheres in $M_r$ representing the same primitive class $[v] \in H_2(M_r;\mathbb Z)$. Any two immersions between $a$ and $b$ have the same rank modulo $r-1$, and this value depends only on the isotopy classes of the spheres.
\end{lemma}

\begin{proof}
Suppose $M_{s,2} \to M_r$ is an immersion from $a$ to $b$. By \cref{immersions exist} there is an immersion $M_{t,2}\to M_r$ from $b$ to $a$. Choose an orientation reversing diffeomorphism $\phi : \partial M_{s,2} \to \partial M_{t,2}$ that identifies the pre-images of $a$ and $b$ under the immersions. Glue the two manifolds using $\phi$ to obtain an immersion $M_{s+t+1} \to M_r$. As this is an immersion of closed 3--manifolds it is a finite-sheeted covering map. Therefore, the rank of the fundamental group satisfies $s+t \equiv 0 \mod r-1$. Thus any immersion from $a$ to $b$ has rank congruent to $-t \mod r-1$.

 It remains to show that this congruence depends only on the isotopy class of the spheres. With $a, b$ and the immersions as before, suppose $b'$ is isotopic to $b$ via an isotopy $b_s$. The boundary gluing gives a lift of $b$ to $M_{s+t+1}$, and the isotopy $b_s$ lifts to an isotopy from $b$ to a lift $\widetilde{b}'$ of $b'$ to $M_{s+t+1}$. By the isotopy extension principle, this lifted isotopy extends to an ambient isotopy of $M_\infty$ \cite{isotopyExtension}*{Section 8.1, Theorem 1.3}.  Hence the embedded copy of $M_{s,2}$ that cobounds $-\widetilde a$ and $b$ is isotopic to a sub-manifold co-bounded by $-\widetilde a$ and $\widetilde b'$. This gives an immersion from $a$ to $b'$ with the same rank as $a$ to $b$. A symmetric argument applies to any sphere isotopic to $a$. We conclude the rank of an immersion from $a$ to $b$ has a fixed value modulo $r-1$ and depends only on the isotopy class.
\end{proof}

We are now in a position to define the coloring. 

\begin{lemma}\label{upper bound}
Suppose $r\ge 2$ and $v$ is an embedded sphere in $M_r$ such that $[v]\in H_2(M_r)$ is primitive. There is a proper vertex coloring 
\[f \colon \Sph_v(M_r) \to \mathbb{Z}/(r-1)\mathbb{Z}\].
Consequently,
\[
\chi\bigl(\Sph_v(M_r)\bigr)\leq r-1.
\]
\end{lemma}

\begin{proof}
For a vertex $a \in \Sph_v(M_r)$, choose an oriented representative such that $[a] = [v]$ and an immersion $M_{h,2} \rightarrow M_r$
from $v$ to $a$. Define
\[
f(a) = h\mod{r-1}.
\]
By \cref{well defined rank between}, this value is independent of the choice of the representative $a$ and of the immersion $M_{h,2}\rightarrow M_r$.  Thus $f$ is well defined.

It remains to show that $f$ is a proper vertex coloring. Let $D_1$ and $D_2$ be
disjoint, nonisotopic oriented embedded spheres representing $[v]$.
Since $[D_1]-[D_2] = 0$, the complement
\[
M_r\setminus N(D_1\cup D_2)
\]
has two components; indeed, if the complement were connected one could construct a curve in $M_r$ with positive algebraic intersection with $[D_1]-[D_2]$. Denote the two components by $W$ and $W'$. Since each of $D_1$ and $D_2$ are non-separating, the boundary of $W$ exactly contains one of $D_1^{\pm}$ and $D_2^{\pm}$, and similarly for $W'$. Chose component labels so that 
$
\partial W=-D_1\cup D_2
$ and
$
\partial W'=D_1\cup -D_2.
$
and note that
$
W\cong M_{t,2}
$ and $
W'\cong M_{s,2}.
$
Since the gluing of $W$ and $W'$ along their boundary spheres is $M_r$, we deduce
$
r=t+s+1,
$
and hence
$
t+s=r-1.
$

Neither $t$ nor $s$ is zero. If either $t=0$ or $s=0$, then
$D_1$ and $D_2$ would be
be isotopic. Therefore
$
0<t,s<r-1.
$

Choose an immersion
$
i\colon M_{h,2}\to M_r
$
from $v$ to $D_1$. Since
$
\partial W=-D_1\cup D_2,
$
we may stack $W$ onto the terminal boundary of this immersion. More
precisely, attach $W$ to $M_{h,2}$ along the boundary sphere
$D_1^-$ and extend $i$ over $W$ using the inclusion
$W\hookrightarrow M_r$. This produces an immersion
$
M_{h+t,2}\to M_r
$
from $v$ to $D_2$. Since $0<t<r-1$, we have
\[
h+t\not\equiv h\pmod{r-1}.
\]
Therefore
\[
f(D_2)\neq f(D_1).
\]
Thus $f$ is a proper vertex coloring with values in
$\mathbb Z/(r-1)\mathbb Z$, and so
\[
\chi(\Sph_v(M_r))\leq r-1.\qedhere
\]
\end{proof}

\section{Sphere pushes and generating sets}\label{generators section}
Throughout the section, assume $r \ge 2$.

The mapping class group of $M$, $\Mod(M)$, acts on the sphere graph $\Sph(M)$ in the natural way.  A mapping class represented by a diffeomorphism $f$ sends the isotopy class $[S]$ of an embedded sphere to the isotopy class $[f(S)]$.

If $S\subset M$ is an embedded sphere, take a regular neighborhood $N(S)\cong S^2\times[0,1]$.  A sphere twist $T_S$ is supported in $N(S)$ and rotates the sphere $S^2\times\{t\}$ through one full turn as $t$ moves from $0$ to $1$, while acting as the identity outside $N(S)$.  Although this diffeomorphism may not be isotopic to the identity, it fixes the isotopy class of every embedded sphere \cite{BrendleBroaddusPutman}*{pp.~1--3}~\cite{Laudenbach}*{Theorem 1.3}. Therefore, sphere twists act trivially on $\Sph(M)$.

The connectivity arguments in \cref{uniqueness section} require descriptions of the generators of $\Mod(M)$ for particular $M_{r,k}$, as well as generators for the pre-image of the Torelli group $IO_r$. In this section we develop the geometric descriptions of the required generating sets. We also establish identities among some of these generators to produce a generating set suited to our purposes. The key lemma for these identities is \cref{twist transfer}, which allows us to lift identities from $\Aut(F_r)$ to a mapping class group, up to a sphere twist term. As sphere twists act trivially on the sphere graph, a relation up to sphere twist is a relation for the induced action on the sphere graph.

We use sphere bases, also known as cut-systems, to describe the geometric generators below. As stated in the introduction, a cut-system in $M_{r,k}$ is a disjoint union of embedded spheres 
$
\mathcal{S} = \{S_1,\dots,S_r\}
$
such that $M_{r,k}\setminus N(\mathcal{S}) = M_{0,k+2r}$ for a sufficiently small regular neighborhood of the system. 

To obtain a cut-system, choose a collection $\{S_1,\ldots, S_r\}$ of disjoint, embedded spheres in $M_{r,k}$ such that, after choosing orientations on each $S_i$, the homology classes 
\[
[S_1],\dots,[S_r]
\]
form a basis of $H_2(M_r;\mathbb{Z})$. Such a collection always exists~\cite{landgraf}*{Lemma B.2}. For a cut-system $\mathcal{S}$, let 
\[
Z = M_{r,k}\setminus N(\mathcal {S})
\]
Then $Z \cong M_{0, k+2r}$. Each sphere in the cut-system gives 2 boundary spheres $S_i^+$ and $S_i^-$ of $Z$.  We refer to $Z$ as the cut-open manifold, and it is the portion of $M_{r,k}$ shown in the cut-sphere diagrams. 

Now, fix a base point $*$ in $Z$. The \emph{geometric free basis} dual to $\mathcal{S}$ is a free basis 

\[
\pi_1(M_{r,k}, *) = F_r = \langle x_1,\dots, x_r\rangle
\]
where $x_i$ is the homotopy class of an oriented loop that, 
in the cut open manifold $Z$, is an oriented curve which goes from the basepoint to $S_i^-$ and then returns to the basepoint via $S_i^+$, disjoint from all other spheres of the cut-system.

Note, the geometric free basis can be realized as a fundamental group of a rose graph embedded in $M_{r,k}$ in the natural way, so in this sense every sphere basis is dual to a rose. Conversely, given a rose embedded in $M_{r,k}$, disjoint embedded representatives of the primitive homology classes Poincar\'{e} dual to each edge give a sphere basis dual to the embedded rose.

\subsection{Point pushes, sphere pushes, and sphere twists}

With the exception of inversions and sphere twists, the diffeomorphisms appearing in our generating sets are constructed by pushing spheres around loops.  For example, a boundary drag moves a boundary sphere around a loop passing through one handle and returns it to its starting position. This is analogous to a disk push in the surface setting, where a boundary circle is moved around a loop by an ambient isotopy. We describe this construction following Landgraf \cite{landgraf}*{Section~4}.

Suppose that $S$ is a boundary sphere of $M_{r,k}$. We cap $S$ with a ball $B$, giving a copy of $M_{r,k-1}$.  Capping allows us to move $B$ though the interior of the manifold by an ambient isotopy.  We push and rotate $B$ around a loop based at a point in its interior, keeping the remaining boundary components fixed, and return $B$ so that the final diffeomorphism fixes $B$ pointwise.  Removing the interior of $B$ then gives a diffeomorphism of $M_{r,k}$ fixing the boundary pointwise.  This is the boundary-pushing construction described by Landgraf \cite{landgraf}*{Section~4}.

Now, suppose that $S$ is an embedded sphere in the interior of $M_{r,n}$.  Cutting along $S$ gives two boundary copies, $S^+$ and $S^-$.  We choose one of these copies and apply the boundary-pushing construction to a path $\gamma$ based at the chosen copy while fixing all other boundary components of the cut open manifold pointwise to obtain a diffeomorphism $\varphi$ of the cut manifold.  Since $\varphi$ fixes the boundary of the cut manifold pointwise, it is compatible with the identification of $S^+$ and $S^-$. Thus $\varphi$ induces a diffeomorphism $\bar{\varphi}$ of $M_{r,k}$. We call $\bar{\varphi}$ the push of $S^\pm$ along $\gamma$. This construction is ambiguous up to a sphere twist in $S^\pm$~\cite{landgraf}*{Theorem 4.1}. As our objective is the action on the sphere graph this ambiguity will be a non-issue in our arguments.

\subsection{Generators for \texorpdfstring{$IO_r$}{IOr}}

Laudenbach's surjection $\Mod(M_r)\to\Out(F_r)$ allows us to represent the Magnus generating set of $IO_r$ as a set of diffeomorphisms of $M_r$ ~\cite{Laudenbach}*{Theorem 4.3}. The representative diffeomorphisms, combined with sphere twists, is a generating set for the pre-image of $IO_r$ in $\Mod(M_r)$. First, recall the Magnus generating set, described with respect to a choice of free basis.

\begin{theorem}[\cite{magnus}]\label{magnus generators}
Let 
\[
F_r = \langle x_1, \dots, x_r \rangle
\]
The group $\IO_r$ is generated by the outer classes of the automorphisms
\begin{align*}
    M_{ij}&:x_i \longmapsto x_jx_ix_j^{-1},\\
    M_{ijk}&: x_i \longmapsto x_i[x_j,x_k],
\end{align*}
where $i,j,k$ are pairwise distinct, $j<k$, and all unlisted basis elements are fixed. 
\end{theorem}

Landgraf defines sphere pushes in $\Mod(M_r)$ that map to the Magnus generators of $IO_r$ in the natural surjection $\Mod(M_r)\to \Out(F_r)$~\cite{landgraf}*{Section~5 and Figures~9--10}. Fix a cut-system and corresponding geometric free basis. With respect to this free basis, the Magnus generator $M_{ij}$ is represented by the \emph{handle drag} $\HD_{ij}$. Informally this is the diffeomorphism that drags the $i$th handle through the $j$th. Precisely, let $\Sigma_i$ be the sphere in the cut-open manifold $Z$ separating $S_i^{\pm}$ from all other spheres of the cut-system. The sphere push of $\Sigma_i$ along a loop representing $x_j$ based on the boundary of $\Sigma_i$ is the \emph{handle drag} $\HD_{ij}$, $\Sigma_i$ and push the resulting boundary sphere around a loop representing $x_j$, shown in schematic in \cref{fig:handle-drag}.
The Magnus generator $M_{ijk}^{-1}$ is represented by the positive commutator drag $\CD^+_{ijk}$. The commutator drag $\CD^+_{ijk}$ is the sphere push of $S_i^+$ around a path representing the commutator $
[x_j,x_k] = x_jx_kx_j^{-1}x_k^{-1}$, see \cref{fig:positive-commutator-drag}.
We summarize this discussion as the following lemma, for ease of cross reference.

\begin{figure}
\begin{subfigure}{0.48\textwidth}
    \centering
\begin{tikzoverlay}[scale=1.0]{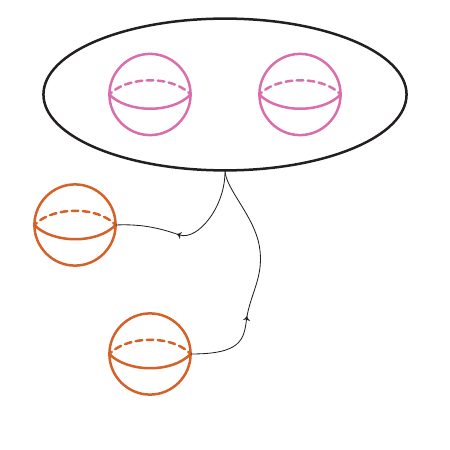}[font=\small]
    \node[figmagenta] at (33.5,73.2) {$S_i^+$};
    \node[figmagenta] at (67.5,73.2) {$S_i^-$};
    \node[figred] at (17.4,44.1) {$S_j^-$};
    \node[figred] at (33.8,15.5) {$S_j^+$};
    \node[] at (82,65.0) {$\Sigma_i$};
\end{tikzoverlay}
    \caption{The handle drag $\HD_{ij}$ is the push of $\Sigma_i$ around a loop representing $x_j$.}
    \label{fig:handle-drag}
    \end{subfigure}
    \hfill
    \begin{subfigure}{0.48\textwidth}
    \centering
    \begin{tikzoverlay}[scale=1.0]{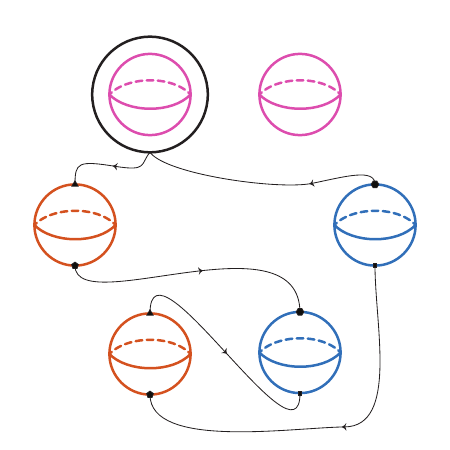}[font=\small]
    \node[figmagenta] at (33.1,73.1) {$S_i^+$};
    \node[figmagenta] at (66.9,73.4) {$S_i^-$};
    \node[figred] at (8.3,40) {$S_j^-$};
    \node[figred] at (25.4,11.1) {$S_j^+$};
    \node[figblue] at (74.6,12.2) {$S_k^-$};
    \node[figblue] at (92.6,42.0) {$S_k^+$};
\end{tikzoverlay}
    \caption{The positive commutator drag $\CD^+_{ijk}$ is the push of a parallel copy of 
    $S_i^+$ around the commutator $[x_j,x_k]$.}
    \label{fig:positive-commutator-drag}
    \end{subfigure}
    \caption{Cut-sphere diagrams indicating the sphere-pushes used for the Magnus generating set of $IO_r$. Cut-spheres not involved in the push are not shown.}
    \label{fig:magnus-generator-picture}
\end{figure}

\begin{lemma}\label{ior generators}
Suppose $r\ge 2$ and let $\mathcal{I}_r \le \Mod(M_r)$ be the pre-image of $\IO_r$ under the natural map $\Mod(M_r)\to \Out(F_r)$. The group $\mathcal{I}_r$ is generated by sphere twists, handle drags $\HD_{ij}$, and the positive commutator drags $\CD^+_{ijk}$.
\end{lemma}

\subsection{Generators for \texorpdfstring{$\Mod(M_{r,2})$}{Mod(M r,2)}}

To construct a generating set for $\Mod(M_{r,2})$ we start by lifting generators for $\Aut(F_r)$ along both the Birman sequence and the quotient by sphere twists. First, Brendle, Broaddus, and Putman give a generating set for $\Mod(M_{r,1})$ in terms of lifts of the Nielsen generating set for $\Aut(F_r)$ and sphere twists.  We then use the Birman exact sequence to account for the additional boundary component. 

Recall that $\Aut(F_r)$ is generated by the Nielsen automorphisms
\[
\iota_i: x_i \mapsto x_i^{-1}, \qquad
\varrho_{ij}: x_i \mapsto x_ix_j, \qquad
\lambda_{ij}: x_i \mapsto x_jx_i
\]
where $i\neq j$ and all unlisted basis elements are fixed ~\cite{lyndon-schupp}*{Proposition I.4.1}.

Fix a cut system for $M_{r,1}$ and a corresponding geometric free basis. We describe diffeomorphisms $I_i$, $R_{ij}$, and $L_{ij}$ representing respectively the automorphisms $\iota_i$, $\varrho_{ij}$, and $\lambda_{ij}$ with respect to the  geometric free basis. We call the representatives \emph{Nielsen diffeomorphisms}. The inversion $I_i$ switches the two ends of the $i$th $S^1\times S^2$ handle and therefore sends $x_i$ to $x_i^{-1}$; formally, in terms of the cut-open manifold $Z$, the inversion $I$ is induced by the diffeomorphism that exchanges $S_i^-$ and $S_i^+$. The diffeomorphisms $R_{ij}$ and $L_{ij}$ are realized by dragging one end of the $i$th handle through the $j$th handle; in the cut-open manifold, dragging $S_i^{\pm}$ along a path representing $x_j$.  Depending on the chosen boundary, the action of the drag on a $x_i$ is $x_ix_j$ or $x_jx_i$.  
Brendle, Broaddus, and Putman illustrate this construction for $R_{12}$~ \cite{BrendleBroaddusPutman}*{Figure~2}. The handle description is made precise by Hatcher and Wahl~\cite{HatcherWahl}*{p.~1303}.

We now describe a generating set for $\Mod(M_{r,2})$. Fix a cut system for $M_{r,2}$ and corresponding geometric free basis. Capping off a fixed boundary sphere $\partial_1 \subset \partial M_{r,2}$ induces a surjection $\Mod(M_{r,2})\to\Mod(M_{r,1})$ (this is a consequence of Landgraf's Birman sequence for these mapping class groups, which we will use extensively~\cite{landgraf}*{Theorem~4.2}). Moreover the inclusion $M_{r,2} \to M_{r,1}$ sends the cut-system and free basis to a cut-system and free basis. We can chose the cap-off sphere $\partial_1$  so that Nielsen diffeomorphism of $\Mod(M_{r,1})$ with respect to this cut-system has support disjoint from the image of $\partial_1$ in $M_{r,1}$. So the collection of Nielsen diffeomorphisms of $\Mod(M_{r,1})$ restricts to a collection of diffeomorphisms in $\Mod(M_{r,2})$, which we also call Nielsen diffeomorphisms, and use the same notation for each.

To complete a description of a generating set for $\Mod(M_{r,2})$ we require the definition of a \emph{boundary drag}. Fix a boundary sphere $\partial_1 \subset \partial M_{r,2}$. There is a unique (up to isotopy) arc $\delta$ from $\partial_1$ to the basepoint. Let $\beta_i$ be the path $\delta\alpha_i\overline{\delta}$, where $\alpha_i$ is a curve representing $x_i$. Let $B_i$ be the mapping class defined by dragging $\partial_1$ around $\beta_i$ that restricts to the identity on $\partial M_{r,2}$.

\begin{proposition}\label{Mod generators}
The group $\Mod(M_{r,2})$ is generated by sphere twists, Nielsen diffeomorphisms, and the boundary drags $B_i$.
\end{proposition}

\begin{proof}
Landgraf develops an appropriate Birman exact sequence for this setting~\cite{landgraf}*{Theorem~4.2}. The quotient of this Birman exact sequence by twists is
\[ 1 \to \pi_1(M_{r,1},\partial_1) \to \Mod(M_{r,2})/\Twist \to \Aut(F_r) \to 1 \]
and moreover, the boundary drags generate the image of
	$\pi_1(M_{r,1},\partial_1)$~\cite{landgraf}*{Theorem~4.2}. By construction the image of the Nielsen diffeomorphisms in $\Mod(M_{r,2})$ surject a generating set for $\Aut(F_r)$. Thus $\Mod(M_{r,2})/\Twist$ is generated by the classes of Nielsen diffeomorphisms and boundary drags. The proposition follows.
\end{proof}

\begin{remark}
Hatcher and Vogtmann describe a version of the ``Birman exact sequence'' for homotopy equivalences of graphs with thorns~\cite{hatcher-vogtmann}*{Proposition 1}. Algebraically this sequence is sufficient for our purposes, as the action of a sphere twist on the sphere graph is trivial. However, as noted in the intorduction, we have found working with diffeomorphisms of 3--manifolds gives a clearer approach than working with homotopy equivalences of graphs.
\end{remark}

In order to discover further identities among the generators, we embed $M_{r,2}$ in $M_{r+1,1}$ by attaching a handle to one boundary component.
Let
\[
\iota:M_{r,2}\hookrightarrow M_{r+1,1}
\]
be the embedding obtained by gluing a copy of $M_{1,1}$ to one boundary component $\partial_1$ of $M_{r,2}$. 
Let
\[
\iota_M:\Mod(M_{r,2})\to\Mod(M_{r+1,1})
\]
denote the homomorphism induced by extension by the identity. Let $S_{r+1}$ be an essential embedded sphere in $M_{1,1}$. Any cut system $\mathcal{S}$ for $M_{r,2}$ can be extended to a cut system for $M_{r+1,1}$ by $\mathcal{S}\cup \{ S_{r+1}\}$. With the basepoint close to $\partial_1$, the loop that passes through $\partial_1$, $S_{r+1}$ and back through $\partial_1$ represents an element $x_{r+1}$ and a geometric free basis $X$ for $\pi_1(M_{r,2})$ corresponding to $\mathcal{S}$ can be extended to a geometric free basis $X\cup \{ x_{r+1}\}$ for $\pi_1(M_{r+1,1})$ corresponding to $\mathcal{S}$. We call $S_{r+1}$ the \emph{additional sphere} and $x_{r+1}$ the \emph{additional generator}.

For every sphere $S\subset M_{r,2}$, extension by the identity sends the sphere twist $T_S$ to the sphere twist $T_{\iota(S)}$. Therefore,
\[
\iota_M\bigl(\Twist(M_{r,2})\bigr) \subseteq \Twist(M_{r+1,1}).
\]

The action of $\Mod(M_{r+1,1})$ on the fundamental group of $M_{r+1,1}$ gives the Laudenbach short exact sequence~\cite{Laudenbach}*{Chapter~III, Theorem~4.3}
\[
1\longrightarrow\Twist(M_{r+1,1})
\longrightarrow\Mod(M_{r+1,1})
\xrightarrow{\rho_{r+1}}
\Aut(F_{r+1})
\longrightarrow 1
\]
Hence $\rho_{r+1}$ induces an isomorphism
\[
\Mod(M_{r+1,1})/\Twist(M_{r+1,1}) \cong \Aut(F_{r+1}).
\]

\begin{lemma}\label{twist transfer}
    Suppose $g,h \in \Mod(M_{r,2})$.  If
    \[
    \rho_{r+1}(\iota_M(g)) = \rho_{r+1}(\iota_M (h)) \in Aut(F_{r+1})
    \]
    then $g$ and $h$ differ by a product of sphere twists in $\Mod(M_{r,2})$.
\end{lemma}

\begin{proof}
If $\rho_{r+1}(\iota_M(g)) = \rho_{r+1}(\iota_M(h))$ then $\iota_M(gh^{-1}) \in \ker\rho_{r+1}$, so $\iota_M(gh^{-1})$ is an element of $\Twist(M_{r+1},1)$. However, the map $\iota_M$ extends $gh^{-1}$ by identity on the additional handle, so $\iota_M(gh^{-1})$ is an element of $\Twist(M_{r+1,1})$ that is identity on the additional handle. This is a characterization of $\Twist(M_{r,2}) \le \Twist(M_{r+1,1}$ and therefore $g$ and $h$ differ by a product of sphere twists in $\Mod(M_{r,2})$.
\end{proof}

We now demonstrate two simplifications of the generating set in \cref{Mod generators}. A Nielsen transformation $\rho_{ij}$ or $\lambda_{ij}$ is \emph{adjacent} if $|i-j| = 1$, and we make the identical definition for Nielsen diffeomorphisms. Our first reduction is to obtain all Nielsen diffeomorphisms from adjacent diffeomorphisms, up to a sphere twist.

\begin{lemma}\label{Nielsen from adjacent Nielsen}
Fix a cut system $\mathcal{S}$ in $M_{r,2}$. For every pair of distinct indices $i,j\in\{1,\dots,r\}$, there exists a product $W_{ij}$ of adjacent right Nielsen diffeomorphisms and their inverses, and sphere twists, and a $\tau \in \Twist$ such that in $\Mod(M_{r,2})$ \begin{align*} R_{ij} &= \circ W_{ij} \\ L_{ij} &= \tau \circ I_i \circ R_{ij}^{-1}\circ I_i.\end{align*}
\end{lemma}

\begin{proof}
We first establish the result for right Nielsen diffeomorphisms by induction on $|i-j|$.

The case $|i-j| = 1$ is immediate.  Suppose that $|i-j|>1$, and let $k$ be the index adjacent to $j$ lying between $i$ and $j$, so that 
\[
|i-k|=|i-j|-1
\qquad\text{and}\qquad
|k-j|=1.
\]

We claim that the following identity holds for some $\tau\in \Twist$:
\[
R_{ij}
=
\tau \circ R_{ik}^{-1}\circ R_{kj}\circ R_{ik} \circ R_{kj}^{-1}.
\]
By \cref{twist transfer}, it is sufficient to show that $R_{ij}$ and the composition on the right-hand side have the same image under $\rho_{r+1}\circ\iota_M$. Both induced automorphisms fix the additional generator $x_{r+1}$, so we check their actions on the geometric free basis $x_1,\dots, x_r$ corresponding to $\mathcal{S}$.

Applying the maps on the right-hand side
sends
\[
x_i
\longmapsto x_i
\longmapsto x_ix_k
\longmapsto x_ix_kx_j
\longmapsto x_ix_j
\]
and
\[
x_k
\longmapsto x_kx_j^{-1}
\longmapsto x_kx_j^{-1}
\longmapsto x_k
\longmapsto x_k.
\]
Every other basis element is fixed. Therefore,
\[
   \rho_{r+1}(\iota_M(R_{ij})) = \rho_{r+1}\left(\iota_M\left(R_{ik}^{-1}\circ R_{kj} \circ  R_{ik} \circ R_{kj}^{-1} \right)\right)
\]
By \cref{twist transfer}, $R_{ij}$ and 
\[
R_{ik}^{-1}\circ R_{kj}\circ R_{ik}\circ R_{kj}^{-1}
\]
differ by a product of sphere twists.

By inductive hypothesis, $R_{ik}$ is a product of adjacent right Nielsen moves, their inverses, and sphere twists while $R_{kj}$ is itself adjacent. Therefore, $R_{ij}$ can be expressed as a product of adjacent right Nielsen moves, their inverses, and sphere twists.  This completes the induction.

Every left Nielsen diffeomorphism can be obtained by \[ L_{ij} = \tau\circ I_i \circ R_{ij}^{-1} I_i \] where $\tau$ is a product of sphere twists; this follows from \cref{twist transfer} and the identity $ \lambda_{ij} = \iota_i \circ \rho_{ij}^{-1}\circ\iota_i$ in $\Aut(F_{r+1})$.
\end{proof}

The second simplification is to obtain every boundary drag from the first boundary drag. Recall that a boundary drag  $B \in \Mod(M_{r,2})$ is defined by the sphere push of $\partial_1$. The sphere $\partial_1$ separates the additional sphere $S_{r+1}$ from the image of any cut system for $M_{r,2}$ under $M_{r,2}\to M_{r+1,1}$. Thus $\iota_M(B)$ is none other than a handle drag of the $r+1$ handle along whatever path defined $B$. This perspective enables us to derive identities among boundary drags using \cref{twist transfer}.

\begin{lemma}\label{boundary drags from first boundary drag}
Fix a cut system $\mathcal{S}$ for $M_{r,2}$ and a boundary component $\partial_1$. For every $j\in\{1,\dots,r\}$, there exists a product $\operatorname{Swap}_{1j}$ of Nielsen diffeomorphisms and their inverses such that the boundary drag $B_j$ and
\[
\operatorname{Swap}_{1j}\circ B_1\circ
\operatorname{Swap}_{1j}^{-1}
\]
differ by a product of sphere twists in $\Mod(M_{r,2})$.
\end{lemma}

\begin{proof}
    Let $x_1,\ldots, x_r$ be the geometric free basis corresponding to $\mathcal{S}$, so that $x_1,\ldots, x_{r+1}$ is the geometric free basis for $M_{r+1,1}$ corresponding to $\mathcal{S}\cup \{ S_{r+1}\}$. In the case $j=1$ take $\operatorname{Swap}_{11} = \operatorname{id}$.  Suppose that $j \neq 1$, and define 
    \[
    \operatorname{Swap}_{1j} = R_{1j}^{-1}\circ I_1\circ L_{j1}^{-1}\circ R_{1j},
    \]
We compute the action on $\pi_1$
    \[
    x_1
    \xmapsto{R_{1j}} x_1x_j
    \xmapsto{L_{j1}^{-1}} x_j
    \xmapsto{I_1} x_j
    \xmapsto{R_{1j}^{-1}} x_j
    \]
    and
    \[
    x_j
    \xmapsto{R_{1j}} x_j
    \xmapsto{L_{j1}^{-1}} x_1^{-1}x_j
    \xmapsto{I_1} x_1x_j
    \xmapsto{R_{1j}^{-1}} x_1.
    \]
    Every $x_k$ with $k\notin\{1,j\}$ is fixed.

    Thus $\operatorname{Swap}_{1j}$ exchanges $x_1$ and $x_j$. Observe $\iota_M(B_j) = HD_{r+1,j}$, the handle drag of the added handle along $x_j$. Thus $\rho_{r+1}\circ\iota_M (B_j)$ is the automorphism $M_{r+1,j}$.  The composition
    \[
    \rho_{r+1}\left(\iota_M\left( \operatorname{Swap}_{1j} \circ B_1 \circ \operatorname{Swap}_{1j}^{-1}\right)\right)
    \]
    maps to the conjugation of $M_{r+1,1}$ by the automorphism swapping $x_1$ and $x_j$. Direct computation shows that this is $M_{r+1,j}$. Therefore, by \cref{twist transfer}
    conjugating the boundary push $B_1$ by $\operatorname{Swap}_{1j}$  differs from $B_j$ by an element of $\Twist(M_{r,2})$.
\end{proof}

\begin{theorem}\label{smaller mod r 2 generators}
Fix a cut system for $M_{r,2}$, corresponding geometric free basis $\{x_1,\ldots, x_r\}$ for $\pi_1(M_{r,2})$, and boundary component $\partial_1$. The group $\Mod(M_{r,2})$ is generated by
\[ \{ T_S, I_i, R_{i(i\pm 1)}, B_1 \}, \] where $T_S$ ranges over all sphere twists, $I_i$ are the inversion Nielsen diffeomorphisms, $R_{i(i\pm 1)}$ are adjacent right Nielsen diffeomorphisms, and $B_1$ is the first boundary drag.
\end{theorem}
\begin{proof}
By \cref{Mod generators}, the collection \[ \{ T_S, I_i, R_{ij}, L_{ij}, B_i\}\] generates $\Mod(M_{r,2})$. \cref{Nielsen from adjacent Nielsen} implies that every $R_{ij}, L_{ij}$ is a composition of adjacent right Nielsen transformations, $I_i$'s and sphere twists. \cref{boundary drags from first boundary drag} demonstrates that $B_i$ is obtained from $B_1$ by a sphere twist and conjugation by a swap, which is a composition of Nielsen transformations. Therefore the generating set from \cref{Mod generators} can be obtained from the set in the theorem statement.
\end{proof}

\section{Uniqueness of the optimal coloring}\label{uniqueness section}

We turn our attention to uniqueness of the coloring.

\begin{proposition}\label{uniqueness proposition}
Suppose $r\ge 2$ and $[v]\in H_2(M_r;\mathbb Z)$ is primitive.
The optimal coloring of $\Sph_v(M_r)$ is unique up to permutation of colors.
\end{proposition}

The proof of this proposition is independent of the construction of the coloring, and relies on a combinatorial condition for the existence of a unique coloring. The \emph{maximal clique graph} $\cliq(G)$ of a graph $G$ is a graph with vertex set the maximal cliques of $G$, and two maximal cliques are joined by an edge if they differ by one vertex. A graph is \emph{pure of dimension $\ell$} if every clique is contained in a $\ell+1$ clique. (This definition is a specialization of the corresponding definition for simplicial complexes to the one-skeleton.) The following result was stated without proof by Gaster, Greene, and Vlamis.

\begin{lemma}[\cite{gaster}*{Proposition 5.13}]\label{unique coloring lemma}
Suppose $G$ is a $n$-colorable graph, pure of dimension $n-1$, with $\cliq(G)$ connected. Then $G$ is uniquely $n$-colorable.
\end{lemma}

\begin{proof}
Consider a maximal clique $\sigma \subset G$. We claim that a proper $n$-coloring $f : V(G)\to \{1,\ldots,n\} $ is uniquely determined by its value on $\sigma$. Indeed, for some other vertex $v\in G$, $v$ is contained in a maximal clique $\tau$. Since $\cliq(G)$ is connected, there exists a path of cliques $\sigma = \sigma_0,\ldots,\sigma_k = \tau$ such that $\sigma_i \cap \sigma_{i+1}$ is an $n-1$ clique. Since each $\sigma_i$ is a clique and $f$ is a proper coloring, the value of $f$ on $\sigma_i$ is determined by its value on an $n-1$ clique. Thus, the value of $f$ on $\sigma_i$ is determined by its value on $\sigma_{i-1}$ for $1\le i \le k$. By induction we conclude the value of $f$ on $\tau$ and hence $v$ is determined by $f$ on $\sigma = \sigma_0$.

Now suppose $f : V(G) \to \{1,\ldots, n\}$ and $g : V(G) \to \{1,\ldots, n\}$ are two proper colorings of $G$. Since $\sigma$ is an $n$-clique, the value of $f$ and $g$ on $\sigma$ differ by a permutation $\pi$ of the label set, so that $\pi\circ f|_\sigma = g|_\sigma$. But as shown above, any coloring of $G$ is determined by its values on $\sigma$. So we conclude $\pi \circ f = g$, that is, $G$ is uniquely $n$-colorable.
\end{proof}

This provides a blueprint to prove \cref{uniqueness proposition}. To prove the required graph connectivity, we use the following lemma from Putman.

\begin{lemma}[\cite{Putman}*{Lemma 2.1}]\label{Putman connectivity}
    Consider a group $G$ acting on a simplicial complex $X$. Fix a basepoint $v\in X^{(0)}$ and a set $S$ of generators for $G$. Assume the following hold
    \begin{enumerate}
        \item For all $v'\in X^{(0)}$, the orbit $G\cdot v$ intersects the connected component of $X$ containing $v'$.
        \item For all $s\in S^{\pm 1}$, there is some path $P_s$ in $X$ from $v$ to $s\cdot v$.
    \end{enumerate}
    Then $X$ is connected.
\end{lemma}

In fact we will prove a slightly stronger statement than item (1) in both applications: we use vertex-transitive actions.

\subsection{Connectivity of the graph of homologous spheres}

We establish connectivity of $\Sph_v(M_r)$ using \cref{Putman connectivity} and the action of $\IO_r$ on $\Sph_v(M_r)$. From \cref{ior generators}, the group $\IO_r$ is generated by the images of handle drags $\HD_{ij}$ and positive commutator drags $\CD^+_{ijk}$.

The transitivity portion of the argument is due to Landgraf~\cite{landgraf}*{Lemma~7.3}, but appears in the course of the proof of a different fact, so we include an explicit formulation of the statement.

\begin{proposition}\label{IO orbits of spheres}
    Suppose $r \ge 2$. Two oriented nonseparating embedded spheres $S,S' \subset M_r$ lie in the same $\IO_r$-orbit if and only if
    \[
    [S] = [S'] \in H_2(M_r;\mathbb{Z})
    \]
    Consequently, for every primitive class $[v] \in H_2(M_r;\mathbb{Z})$, $\IO_r$ acts transitively on $\Sph_v(M_r)$.
\end{proposition}

\begin{proof}
    Every element of $\IO_r$ acts trivially on $H_2(M_r;\mathbb Z)$, so the homology class of an oriented sphere is constant on each $\IO_r$-orbit.

    Conversely, suppose that $[S] = [S'] = v$.  Extend $v = v_1$ to a basis 
    \[
    v_1, \dots, v_r
    \]
    of $H_2(M_r;\mathbb{Z})$.  Landgraf's sphere-basis result \cite{landgraf}*{Lemma~B.2} allows us to extend $S = S_1$ and $S' = S_1'$ to sphere bases
    \[
    S_1, \dots, S_r
    \qquad\text{and}\qquad S_1', \dots, S_r'
    \]
    such that 
    \[
    [S_\ell]= [S_\ell']=v_\ell
    \qquad\text{for every } 1 \le \ell \le r
    \]

    Cutting $M_r$ along either sphere basis produces a $3$-sphere with $2r$ open balls removed.  Matching the correspondingly labeled boundary spheres therefore gives an orientation-preserving diffeomorphism of $M_r$ carrying each $S_\ell$ to $S_\ell'$. This diffeomorphism fixes the basis $v_1, \dots, v_r$ of homology, so its image in $\Out(F_r)$ belongs to $\IO_r$.  In particular it carries the oriented sphere $S$ to $S'$.

    Finally, every primitive class is represented by a nonseparating sphere \cite{landgraf}*{Lemma~B.2}.  Hence each $\Sph_v(M_r)$ is a single $\IO_r$-orbit.
\end{proof}

Now that we have proved transitivity, we verify the remaining hypothesis needed for \cref{Putman connectivity}.

\begin{lemma}\label{homologous sphere graph connected}
Suppose $r \ge 2$ and $v$ is an embedded essential sphere such that $[v] \in H_2(M_r,\mathbb{Z})$ is a primitive class. The sphere graph $\Sph_v(M_r)$ is connected.
\end{lemma}

\begin{proof}
We apply \cref{Putman connectivity}. 
As before, there is an extension $S_1=v$ to a basis
\[
[S_1],\dots,[S_r]
\]
of $H_2(M_r;\mathbb Z)$, represented by essential spheres 
\[
S_1,\ldots,S_r,
\]
and the collection forms a cut system for $M_r$~\cite{landgraf}*{Lemma B.2}. Choose the corresponding geometric free basis $x_1,\ldots,x_r$ of $\pi_1(M_r)$.

By \cref{IO orbits of spheres}, the action of $\IO_r$ on $\Sph_v(M_r)$ is transitive. The group $\IO_r$ is generated by the images of handle drags $\HD_{ij}$ and the positive commutator drags $\CD^+_{ijk}$, by \cref{ior generators}. Therefore, to apply \cref{Putman connectivity}, it remains to show that $v_1$ and $s(v_1)$ lie in the same connected component of $\Sph_v(M_r)$ for the diffeomorphisms that represent a generating set of $IO_r$.

We first consider the handle drags $\HD_{ij}$.  In the cut-sphere model, with respect to the cut system $S_1, \dots, S_r$, the $i$th handle is pushed around the loop representing $x_j$. The dragging path crosses $S_j$ and can be chosen disjoint from every other sphere in the chosen sphere basis. Moreover, the map is extended by the identity over the $i$th handle piece. Thus,
\[
\HD_{ij}(S_\ell)=S_\ell
\qquad\text{whenever }\ell\neq j.
\]
In particular, if $j\neq 1$, then
\[
\HD_{ij}(S_1)=S_1.
\]

It remains to consider the handle drags
\[
\HD_{i1},
\qquad i\neq 1.
\]
Examine the cut-sphere diagram for $M_r$ with respect to the cut system $S_1,\ldots, S_r$. The loop representing $x_1$ appears in the diagram as an arc joining $S_1^-$ to $S_1^+$ via the basepoint.  The handle drag $\HD_{i1}$ carries the $i$th handle through $S_1$.  As illustrated in \cref{fig:handle-drag-Mi1}, the image $\HD_{i1}(S_1)$ can be represented by a sphere disjoint from $S_1$. Hence either $\HD_{i1}(S_1)$ is isotopic to $S_1$ or joined to $S_1$ by an edge in $\Sph_v(M_r)$.

\begin{figure}
    \centering
\begin{tikzoverlay}[scale=1.0]{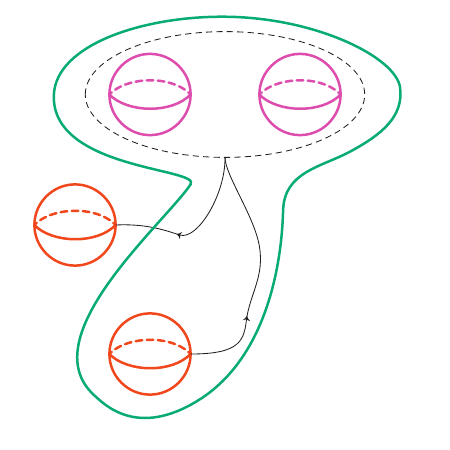}[font=\small]
    \node[figgreen] at (74,43.3) {$\HD_{i1}(S_1)$};
    \node[figmagenta] at (33.3,73.2) {$S_i^+$};
    \node[figmagenta] at (67.4,73.2) {$S_i^-$};
    \node[figred] at (16.9,44.4) {$S_1^-$};
    \node[figred] at (33.6,15.6) {$S_1^+$};
\end{tikzoverlay}
    \caption{The handle drag $\HD_{i1}$ applied to $S_1$. The image $HD_{i1}(S_1)$
    is isotopic to a sphere disjoint from $S_1$.}
    \label{fig:handle-drag-Mi1}
\end{figure}

We next consider the positive commutator drags $\CD^+_{ijk}$. By construction, the pushing path meets only $S_j$ and $S_k$, entering their boundary copies in the order
\[
S_j^{-}, \quad S_k^{-}, \quad S_j^{+}, \quad S_k^{+},
\]
and therefore,
\[
\CD_{ijk}^{+}(S_\ell)=S_\ell
\qquad\text{whenever }\ell\notin\{j,k\}.
\]
In particular, $\CD_{ijk}^+$ fixes $S_1$ unless $1 \in \{j,k\}$.

Since $j<k$, the only remaining case to analyze is
\[
\CD^+_{i1k},
\qquad i,k\neq 1.
\]
It remains to show that $S_1$ and $\CD^+_{i1k}(S_1)$ lie in the same connected component. We claim they are connected by a path of length two. 

We conduct our analysis in a cut-sphere diagram with respect to the cut system given by the $S_i$.  Consider a sphere $T$ constructed from two disks. The first disk has boundary on $S_k^+$ and separates $S_1^-$ from the remainder of the cut-spheres in the diagram. The second disk has boundary on $S_k^-$ and separates $S_i^-$ and $S_i^+$ from the remainder of the cut-spheres in the diagram. These disks can be chosen disjoint from both $S_1$ and $\CD^+_{i1k}(S_1)$, as shown in \cref{fig:commutator-drag}. Therefore the disks glue to a sphere $T$ in $M_{r}$ disjoint from $S_1$ and $CD^+_{i1k}(S_1)$.

Moreover, by construction $T$ is separating in $M_r \setminus N(S_1)$. Consequently,
the region between $T$ and $S_1^+$ is compact
submanifold $W\subset M_r$, oriented so that
\[
\partial W=-T\sqcup S_1.
\]
The inclusion $W\hookrightarrow M_r$ is an embedding, and hence an
immersion from $S_1$ to $T$. In particular,
\[
[T]=[S_1]=v.
\]

\begin{figure}
    \centering

    \begin{subfigure}[t]{0.3\textwidth}
        \centering
\begin{tikzoverlay}[scale=1.0]{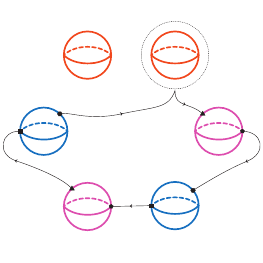}[font=\tiny]
    \node[figred] at (32.9,65.5) {$S_i^-$};
    \node[figred] at (61.9,67.1) {$S_i^+$};
    \node[figblue] at (7.1,60.7) {$S_k^-$};
    \node[figblue] at (75.4,9.9) {$S_k^+$};
    \node[figmagenta] at (83.7,64.3) {$S_1^-$};
    \node[figmagenta] at (23.4,9.5) {$S_1^+$};
\end{tikzoverlay}
\caption{The path of the commutator drag, simplified.}
        \label{fig:commutator-T}
    \end{subfigure}
    \begin{subfigure}[t]{0.3\textwidth}
        \centering
\begin{tikzoverlay}[scale=1.0]{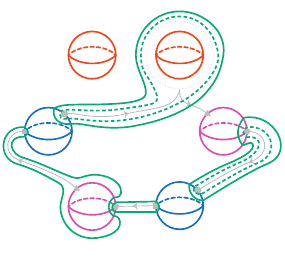}[font=\tiny]
    \node[figred] at (32.0,61.0) {$S_i^-$};
    \node[figred] at (60.1,61.6) {$S_i^+$};
    \node[figblue] at (8.1,56.0) {$S_k^-$};
    \node[figblue] at (70.3,7.8) {$S_k^+$};
    \node[figmagenta] at (82.8,58.2) {$S_1^-$};
    \node[figmagenta] at (20.4,7.1) {$S_1^+$};
    \node[figgreen] at (45.8,43.2) {$\CD^+_{i1k}(S_1)$};
\end{tikzoverlay}
        \caption{The image $CD^+_{i1k}(S_1)$.}
        \label{fig:S1-T}
    \end{subfigure}
    \begin{subfigure}[t]{0.35\textwidth}
        \centering
\begin{tikzoverlay}[scale=1.0]{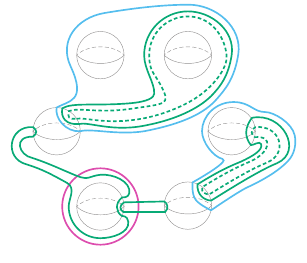}[font=\tiny]
    \node[black!50] at (32.0,56.0) {$S_i^-$};
    \node[black!50] at (9.5,51.0) {$S_k^-$};
    \node[black!50] at (32.0,18.4) {$S_1^+$};
    \node[black!50] at (66.3,7.8) {$S_k^+$};
    \node[black!50] at (74.8,42.5) {$S_1^-$};
    \node[black!50] at (60.1,56.0) {$S_i^+$};
    \node[figmagenta] at (20.4,7.1) {$S_1$};
    \node[figcyan] at (86.7,53.7) {$T$};
\end{tikzoverlay}
        \caption{Realization of the path.}
        \label{fig:T-commutator}
    \end{subfigure}
    \caption{Construction of a length-two path $S_1, T, \CD^+_{i1k}(S_1)$ in $\Sph_v(M_r)$.}
    \label{fig:commutator-drag}
\end{figure}

Since $\CD^+_{i1k}\in \IO_r$, the sphere
$\CD^+_{i1k}(S_1)$ also represents $v$. Consequently,
\[
S_1,T,\CD^+_{i1k}(S_1)
\]
is a path of length two in $\Sph_v(M_r)$.

Since the $\IO_r$ action is transitive and a basepoint is connected to its images under a generating set, \cref{Putman connectivity} implies that $\Sph_v(M_r)$ is connected.
\end{proof}

\subsection{Connectivity of cliques of homologous spheres}

The connectivity of $\cliq(\Sph_v(M_r))$ passes through the connectivity of a particular sphere graph of a cut manifold. Let $\Sph_\partial (M_{r,2})$ be the induced subgraph of $\Sph (M_{r,2})$ on the set of spheres homologous to a boundary sphere. We first show the maximal clique graph $\cliq(\Sph_\partial(M_{r,2})$ is connected, again using \cref{Putman connectivity} and the action of $\Mod(M_{r,2})$. Once this auxiliary result is obtained, the connectivity of the clique graph follows a ``clique tracking'' strategy similar to Gaster, Greene, and Vlamis'~\cite{gaster}*{Proposition 5.13}.

Fix a standard maximal clique in $\Sph_\partial(M_{r,2})$ using the construction from the proof of \cref{lower bound}, to decompose $M_{r,2}$ into $r$ copies of $M_{1,2}$.  The decomposition spheres
\[
K_0 = \{v_1,\dots,v_{r-1} \}
\]
are pairwise disjoint and, by the same homology calculation as in the proof of \cref{lower bound}, represent the peripheral class of $M_{r,2}$; see \cref{fig:block-decomposition}.

The clique $K_0$ is maximal. Cutting along its spheres produces $r$ copies of $M_{1,2}$.  Any peripheral sphere disjoint from $K_0$ lies in one of these components and is parallel to a boundary there.  It is therefore either isotopic to a sphere already in $K_0$ or boundary-parallel in $M_{r,2}$. We call $K_0$ the standard maximal clique.

We will show that every maximal clique in $\Sph_\partial(M_{r,2})$ is carried to the standard maximal clique $K_0$ by a diffeomorphism of $M_{r,2}$ fixing the boundary pointwise.

\begin{lemma}\label{peripheral clique transitive action}
The action of $\Mod(M_{r,2})$ on $\cliq(\Sph_\partial(M_{r,2}))$ is transitive.
\end{lemma}

\begin{proof}
First we show that if $K$ is a maximal clique, then $M_{r,2} \setminus K$ is a union of $r$ copies of $M_{1,2}$. We proceed by induction on $r$. There are no essential and non-peripheral separating spheres in $M_{1,2}$ that represent the boundary class, so $\Sph_\partial(M_{1,2})$ is empty and we take the convention the empty graph is connected. The next base case is $r=2$. Let $S$ be a non-peripheral sphere in the boundary homology class in $M_{2,2}$. Cutting along $S$ splits $M_{2,2}$ into $M_{1,2}\cup M_{1,2}=\cup^2_{i=1}M_{1,2}$, so $S$ is a maximal clique in $\Sph_\partial(M_{2,2})$, and this base case holds.

Next, consider $M_{r,2}$ for some $r>2$. Suppose that for all $2 \le \ell <r$, cutting $M_{\ell, 2}$ along the spheres in a maximal clique in $\cliq(\Sph_\partial (M_{\ell,2})$ yields $\cup^\ell M_{1,2}$. 

Choose a sphere $S$ from the maximal clique $K$ of $\Sph_\partial(M_{r,2})$.  Since $S$ is homologous to the boundary, it separates the two components of the boundary of $M_{r,2}$. After cutting along $S$, we obtain $M_{p,2}$ and $M_{q,2}$ with $p+q = r$. Further, $1 \le p,q$ since $S$ is non-peripheral. 

The spheres of $K\setminus S$ form maximal cliques in their respective components.  The inductive hypothesis shows that cutting along the remaining spheres in $K$ yield a union of $p$ and $q$ copies of $M_{1,2}$.

Thus, for a given maximal clique $K$, the components of the decomposition $M_{r,2} \setminus K$ are diffeomorphic to those for the complement of the standard maximal clique $M_{r,2}\setminus K_0$. Each complement comes with an identification pairing of boundary spheres that recovers $M_{r,2}$. Moreover, the components of each complement are diffeomorphic via a diffeomorphism that fixes all boundaries pointwise and agrees with the boundary pairing and on $\partial M_{r,2}$. By gluing we obtain a representative of an element of $\Mod(M_{r,2})$ that takes $K$ to $K_0$. Therefore, the action is transitive.
\end{proof}

\begin{lemma}\label{peripheral sphere clique graph connected}
The graph $\cliq(\Sph_\partial (M_{r,2}))$ is connected.
\end{lemma}

\begin{proof}
We show $\cliq(\Sph_\partial (M_{r,2}))$ is connected using \cref{Putman connectivity} by considering $\Mod(M_{r,2})$ acting on $\cliq(\Sph_\partial (M_{r,2}))$. 

By \cref{peripheral clique transitive action}, the action of $\Mod(M_{r,2})$ on $\cliq(\Sph_\partial(M_{r,2}))$ is transitive.

We now verify the second condition of \cref{Putman connectivity}. To do so, fix a cut-system $\{S_1,\ldots, S_r \}$ disjoint from the standard maximal clique $K_0$ such that each $S_i$ is essential in one component of $M_{r,2}\setminus N(K_0)$. Fix a corresponding geometric free basis $x_1,\ldots x_r$, and note that each geometric free basis curve is freely homotopic to the generator of the fundamental group of a component of $M_{r,2}\setminus N(K_0)$. Let $s$ be one of the generators from the generating set in \cref{smaller mod r 2 generators}:
\[
\{T_S,I_i,R_{i(i\pm 1)},B_1\}
\]
of $\Mod(M_{r,2})$ corresponding to the chosen cut system.  We must construct a path in the clique graph from $K_0$ to $sK_0$ for each type of generator.

We begin with the generators that preserve the standard clique.  A sphere twist acts trivially on the isotopy class of every embedded sphere, and therefore
\[
T_SK_0=K_0.
\]
Likewise, the inversion $I_i$ is isotopic to a diffeomorphism supported in the $i$th $M_{1,2}$ block. This diffeomorphism reverses the direction of the core curve of $M_{1,2}$ (freely homotopic to $x_i$ while fixing the boundary spheres of $M_{1,2}$.  It therefore fixes every sphere of the standard clique, and hence
\[
I_iK_0=K_0.
\]
Thus sphere twists and inversions satisfy the second condition of \cref{Putman connectivity} by the degenerate path at $K_0$.

We next consider an adjacent right Nielsen diffeomorphism $R_{ij}$.
Let $\sigma \in K_0$ be the sphere separating the $M_{1,2}$ components containing $S_i$ and $S_j$. Such a sphere exists because $R_{ij}$ is adjacent.  The diffeomorphism realizing $R_{ij}$ is the sphere push of one side of $S_i$ through $S_j$ and is isotopic to a diffeomorphism supported on the $M_{1,2}$ components adjacent to $\sigma$ and fixing the boundary of the union of these components pointwise.

It follows that $R_{ij}$ fixes every sphere of $K_0$ other than $\sigma$, while replacing $\sigma$ by its image $R_{ij}(\sigma)$.  Therefore
\[
K_0\cap R_{ij}K_0=K_0\setminus\{S_k\},
\]
so $K_0$ and $R_{ij}K_0$ are joined by an edge in the clique graph.  Applying $R_{ij}^{-1}$ to this edge shows that $K_0$ and $R_{ij}^{-1}K_0$ are also joined by an edge.

For the boundary drag generator $B_1$, observe that $B_1$ is entirely supported in the first block, and therefore fixes every sphere in the clique. Thus, $B_1K_0 = K_0$.

We have therefore constructed a path from $K_0$ to $sK_0$ for every generator $s$ of $\Mod(M_{r,2})$. The result now follows from \cref{Putman connectivity}.
\end{proof}

We now deduce the connectivity for $\Sph_v(M_r)$ from the previoius lemma. The connection passes through identifying the link of a sphere in $\Sph_v(M_r)$, which also allows us to show that the sphere graph is of pure dimension.

\begin{lemma}\label{inclusion is link}
For a sphere $a \in \Sph_v(M_r)$, the inclusion $M_r\setminus N(a) \to M_r$ induces an isomorphism
\[ \operatorname{Lk}(a) \cong \Sph_\partial(M_{r-1,2}). \]
\end{lemma}
\begin{proof}
Choose an oriented representative $A \subset M_r$ of $a$, and let 
\[
N = M_r \setminus N(A)\cong M_{r-1,2}.
\]
A sphere representing a vertex of $\operatorname{Lk}_{\Sph_v(M_r)}(a)$ has a representative disjoint from $A$, and hence determines a sphere in $N$.  Since it is homologous to $A$, the resulting sphere separates the two boundary components of $N$, so it represents a vertex of $\Sph_\partial(N)$.

Conversely, after regluing the two boundary components of $N$, every sphere that is a vertex of $\Sph_\partial (N)$ in $N$ becomes a sphere in $M_r$ that is disjoint from $A$ and represents the same oriented homology class as $A$. These constructions preserve isotopy classes and disjointness and are inverse to one another. Therefore, they induce the desired isomorphism. 
\end{proof}

\begin{lemma}\label{homologous sphere graph pure dimension}
The sphere graph $\Sph_v(M_r)$ is pure of dimension $r-2$.
\end{lemma}

\begin{proof}
Consider a clique $\sigma \subset \Sph_v(M_r)$, and pick a sphere $A$ representing a vertex $a$ of $\sigma$. Cut $M_r$ along $A$ to obtain a copy of $M_{r-1,2}$. Since $\sigma \setminus \{a\}$ has representatives disjoint from $A$, it corresponds to a clique in $\Sph_\partial(M_{r-1,2})$. Complete $\sigma\setminus \{a\}$ to a maximal clique $\sigma '$ in $\Sph_\partial(M_{r-1,2})$. \cref{peripheral clique transitive action} shows all such cliques are of size $r-2$. By \cref{inclusion is link} $\sigma'$ is an $r-2$ clique in the link of $a$. Thus the subgraph $\tau$ induced by $\sigma'$ and $a$ is an $r-1$ clique containing $\sigma$.
\end{proof}

This brings us to the connectivity lemma, which is the remaining piece of the proof of \cref{uniqueness proposition}.

\begin{lemma}\label{homologous sphere clique graph connected}
The graph $\cliq(\Sph_v(M_r))$ is connected.
\end{lemma}

\begin{proof}
Let $K$ and $K'$ be maximal cliques in $\Sph_v(M_r)$, and choose vertices $a_0 \in K$ and $a_m \in K'$.  By \cref{homologous sphere graph connected}, there is an edge path 
\[
a_0,a_1,\dots, a_m
\]
in $\Sph_v(M_r)$. By \cref{homologous sphere graph pure dimension}, each edge $\{a_{i-1},a_i\}$ is contained in a maximal clique. Choose one such clique $K_i$.  Set
\[
K_0 = K
\]
and 
\[
K_{m+1} = K'
\]
Thus, for each $0 \le i \le m$, the cliques $K_i$ and $K_{i+1}$ contain $a_i$.

Fix $0 \le i \le m$, and choose an oriented representative $A_i \subset M_r$ of $a_i$.  Cutting $M_r$ along $A_i$ produces a manifold
\[
N_i \cong M_{r-1,2}.
\]

By \cref{inclusion is link}, the inclusion $N_i \to M_r$ incudes an isomorphism
\[
\operatorname{Lk}_{\Sph_v(M_r)}(a_i) \cong \Sph_\partial(N_i)
\]
Under this isomorphism, $K_i \setminus \{a_i\}$ and $K_{i+1} \setminus \{a_i\}$ are maximal cliques in $\Sph_\partial(N_i)$. By \cref{peripheral sphere clique graph connected}, there is a path
\[
C_0,C_1,\dots,C_\ell
\]
in $\cliq(\Sph_\partial(N_i))$ with $C_0 = K_i\setminus \{a_i\}$ and $C_\ell = K_{i+1} \setminus \{a_i\}$.

Each $C_j \cup \{a_i\}$ is a maximal clique in $\Sph_v(M_r)$, and codimension-one adjacency of consecutive cliques is preserved after adjoining $a_i$. Thus
\[
C_0\cup\{a_i\},C_1\cup\{a_i\},\dots,C_\ell\cup\{a_i\}
\]
is a path in $\cliq(\Sph_v(M_r))$ from $K_i$ to $K_{i+1}$.

Repeating this construction for $i = 0, \dots, m$ and concatenating the resulting paths gives a path from $K$ to $K'$.  Therefore $\cliq(\Sph_v(M_r))$ is connected. 
\end{proof}

This completes the verification of the hypotheses of \cref{unique coloring lemma}.

\begin{proof}[Proof of \cref{uniqueness proposition}]
By \cref{existence proposition}, \cref{homologous sphere graph pure dimension} and \cref{homologous sphere clique graph connected}, the graph $\Sph_v(M_r)$ satisfies the hypotheses of \cref{unique coloring lemma} for $r-1$ colors.
\end{proof}

\end{document}